\documentclass[a4paper, reqno, 11pt]{amsart}

\usepackage[english]{babel}
\usepackage{amsmath}
\usepackage{mathrsfs}
\usepackage{amssymb}
\usepackage{amsthm}
\usepackage{enumerate}
\usepackage{ifthen}
\usepackage{bbm}
\usepackage{color}
\provideboolean{shownotes} 
\setboolean{shownotes}{true}
\usepackage{hyperref}
\usepackage{graphicx}
\usepackage{pgfplots}
\usepackage{tikz,tikz-3dplot} 
\usepackage{mathtools}
\usepackage{caption}[2004/07/16]

\pgfplotsset{compat=1.8}

\newcommand{\margnote}[1]{
\ifthenelse{\boolean{shownotes}}%
{\marginpar{\raggedright\tiny\texttt{#1}}}%
{}%
}

\newcommand{\hole}[1]{
\ifthenelse{\boolean{shownotes}}%
{\begin{center} \fbox{ \rule {.25cm}{0cm}
\rule[-.1cm]{0cm}{.4cm} \parbox{.85\textwidth}{\begin{center}
\texttt{#1}\end{center}} \rule {.25cm}{0cm}}\end{center}}
{}
}
\newtheorem{thm}{Theorem}[section]
\newtheorem*{thm1}{Theorem 1.1}
\newtheorem*{thm2}{Theorem 1.2}
\newtheorem{prop}[thm]{Proposition}

\newtheorem{cor}[thm]{Corollary}

\newtheorem{exm}[thm]{Example}
\theoremstyle{definition}
\newtheorem{defn}[thm]{Definition}

\newcommand{\e}{\varepsilon}		       
\newcommand{\R}{\mathbb{R}}

\newcommand{\N}{\mathbb{N}}

\newcommand{\dive}{\mathop{\mathrm {div}}}

\newcommand{\sgn}{\text{sgn}}
\newcommand{\x}{X_1^\varepsilon}
\newcommand{\y}{X_2^\varepsilon}
\newcommand{\z}{X_3^\varepsilon}
\newcommand{\de}{\mathrm{d}}

\numberwithin{equation}{section}

\begin{document}

\title[Smooth approximation is not a selection principle]{Smooth approximation is not a selection principle for the transport equation with rough vector field}

\author[G. Ciampa]{Gennaro Ciampa}
\address[G. Ciampa]{GSSI - Gran Sasso Science Institute\\ Viale Francesco Crispi 7 \\67100 L'Aquila \\Italy \& Department Mathematik und Informatik\\ Universit\"at Basel \\Spiegelgasse 1 \\CH-4051 Basel \\ Switzerland}
\email[]{\href{gennaro.ciampa@}{gennaro.ciampa@gssi.it},\href{gennaro.ciampa@2}{gennaro.ciampa@unibas.ch}}

\author[G. Crippa]{Gianluca Crippa}
\address[G. Crippa]{Department Mathematik und Informatik\\ Universit\"at Basel \\Spiegelgasse 1 \\CH-4051 Basel \\ Switzerland}
\email[]{\href{gianluca.crippa@}{gianluca.crippa@unibas.ch}}

\author[S. Spirito]{Stefano Spirito}
\address[S. Spirito]{DISIM - Dipartimento di Ingegneria e Scienze dell'Informazione e Matematica\\ Universit\`a degli Studi dell'Aquila \\Via Vetoio \\ 67100 L'Aquila \\ Italy}
\email[]{\href{stefano.spirito@}{stefano.spirito@univaq.it}}

\begin{abstract}
In this paper we analyse the selection problem for weak solutions of the transport equation with rough vector field. We answer in the negative the question whether solutions of the equation with a regularized vector field converge to a unique limit, which would be the selected solution of the limit problem. To this aim, we give a new example of a vector field which admits infinitely many flows. Then we construct a smooth approximating sequence of the vector field for which the corresponding solutions have subsequences converging to different solutions of the limit equation.
\end{abstract}

\maketitle

\section{Introduction}
Consider the Cauchy problem for the transport equation
\begin{equation}
\begin{cases}
\partial_t u(t,x)+ b(t,x) \cdot \nabla u(t,x)=0, \\
u|_{t=0}=u_0, 
\end{cases}
\label{eq:te}
\end{equation}
where $(t,x)\in (0,T)\times\R^{d}$ are the independent variables, with $T<\infty$, $b:(0,T)\times\R^{d}\to\R^{d}$ is a given divergence-free vector field and $u_{0}:\R^{d}\to\R$ is a given initial datum. The equation \eqref{eq:te} is connected with the system of ordinary differential equations 
\begin{equation}
\begin{dcases}
\frac{\de}{\de t}X(t,x)=b(t,X(t,x)), \\
X(0,x)=x,
\end{dcases}
\label{eq:ode}
\end{equation} 
where the unknown $X:(0,T)\times\R^{d}\to\R^{d}$ is referred to as the {\em flow} of the vector field $b$. The aim of this paper is to study possible selection criteria for the uniqueness of solutions of \eqref{eq:te} in a setting of low regularity.\\
The transport equation \eqref{eq:te} is classically well-posed when the vector field and the initial datum are smooth. Specifically, assume that the vector field $b$ is globally Lipschitz, then existence and uniqueness of smooth solutions with Lipschitz initial data can be proved by exploiting the connection between \eqref{eq:te} and \eqref{eq:ode} and the fact that \eqref{eq:ode} is well-posed. However, mainly due to the applications to fluid dynamics and conservation laws, the setting of smooth vector fields is too restrictive and a theory in weaker regularity settings has been developed in the last decades. In this paper we give a new example of nonuniqueness and we provide a counterexample to a possible selection principle of a unique solution of \eqref{eq:te} and \eqref{eq:ode} with rough vector fields. In order to set the problem and to explain exactly our result we provide a brief overview of relevant previous results on the analysis of \eqref{eq:te}.
\subsection*{A short review of some previous results}
The theory of existence and uniqueness of solutions of \eqref{eq:te} and \eqref{eq:ode} in a smooth setting is based on the method of characteristics. Loosely speaking, suppose $b$ is a globally Lipschitz and divergence-free vector field, then the Cauchy-Lipschitz theorem ensures the existence of a unique measure-preserving flow $X$ solution of \eqref{eq:ode}. Then, $u(t,x)=u_{0}((X(t,\cdot))^{-1}(x))$ is a smooth solution of \eqref{eq:te}. Finally, a simple estimate of the difference of two possible solutions of \eqref{eq:te} starting with the same initial datum implies that such $u$ is also the unique solution of \eqref{eq:te}. In a nonsmooth setting the situation is much more complex.  The existence of bounded distributional solutions can be obtained by a simple approximation procedure requiring only integrability hypotesis on $b$. While the existence is obtained by standard arguments, the uniqueness of distributional solutions is much more difficult and require additional assumptions on the vector field. The first result in this direction is due to DiPerna and Lions \cite{DPL}, where the uniqueness of distributional solutions of \eqref{eq:te} is proved under the hypothesis that $b$ has Sobolev regularity and bounded divergence. The result in \cite{DPL} has been extended in the highly non trivial case of $BV$ vector fields with bounded divergence by Ambrosio in \cite{Am}. Furthermore, Bianchini and Bonicatto in \cite{BB} have recently shown uniqueness in the case of a nearly incompressible $BV$ vector field, without assumptions on the divergence, giving a positive answer to the Bressan's compactness conjecture, see \cite{B}.
In all these uniqueness results the key point is to assume a control on one full derivative of the vector field in some weak sense. Several counterexamples to the uniqueness are available in the case of less regular vector fields. In particular, based on a counterexample of Aizemann \cite{Ai}, Depauw in \cite{DeP} showed an example of a divergence-free vector field $b\in L^{1}((\e,T);BV(\R^{2}))$ for any $\e>0$, but not in $L^{1}((0,T);BV(\R^{2}))$, for which the Cauchy problem \eqref{eq:te} with $u_0=0$ admits a nontrivial bounded solution. In \cite{ABC14, ABC13} the authors give an example of an autonomous divergence-free vector field which belongs to $C^{0,\alpha}(\R^2)$ for every $\alpha<1$, for which uniqueness of bounded solutions fails. Exploiting convex integrations techniques, examples of nonuniqueness of bounded solutions of \eqref{eq:te} are provided also in \cite{CGSW} for bounded and divergence-free vector fields. Nonuniqueness of weak solutions with integrability lower than the one considered in \cite{DPL} is shown in \cite{MS,MS2,MS3}.
Finally, a very important counterexample for the purpose of this paper is the one of DiPerna and Lions, \cite{DPL}, where they consider a divergence-free vector field $b\in W^{s,1}_{\mathrm{loc}}(\R^{2})$ for every $s<1$ which admits two different measure preserving flows.
\subsection*{The problem of selection}
In order to state and motivate the counterexample presented in this paper, we illustrate in some more detail the proof of existence of bounded distributional solutions of the problem \eqref{eq:te}. We assume that the datum $u_0$ is smooth since this assumption does not affect the analysis. Suppose that $b$ is a divergence-free vector field in $L^p((0,T);L^p(\R^d))$. A very common and natural approximation of the transport equation is obtained by considering a sequence of smooth vector fields $\{b_{\e}\}_{\e}$ with divergence uniformly bounded in $\e$ and converging strongly in $L^p((0,T);L^p(\R^d))$ to $b$. 
Then, since for each fixed $\e$ the vector field $b_{\e}$ is smooth, there exists a unique solution $u_{\e}$ of the Cauchy problem
\begin{equation}\label{eq:tee}
\begin{cases}
\partial_t u^\e+b_\e\cdot\nabla u^\e=0,\\
u_{\e}|_{t=0}=u_0.
\end{cases}
\end{equation}
Using the explicit formula for smooth solutions and by standard compactness arguments, {\em up to a subsequence}, there exists at least a weak-star limit $u\in L^{\infty}((0,T);L^\infty(\R^d))$, which is a distributional solution of \eqref{eq:te}. Of course, as for all compactness arguments, the previous proof gives no information on the uniqueness since there is a passage to subsequences. A natural question is therefore the following:
\begin{itemize}
\item[(Q1)]\emph{Does the approximation procedure obtained by smoothing the vector field select a unique solution of \eqref{eq:te}?}
\end{itemize}
In this paper we give a negative answer to the above question in the three dimensional case. First of all consider the set 
$$
\mathcal{I}=\left\{f\in C^\infty(\R^3): f=f(r,\theta,z) \mbox{ with } f(r,\theta_1,z)\neq f(r,\theta_2,z) \mbox{ if }\theta_1\neq\theta_2 \right\},
$$
where $(r,\theta,z)$ denotes the cilindrical coordinates in $\R^3$.
Our main theorem is the following:
\begin{thm}\label{teo:main1}
There exist an autonomous divergence-free vector field $b\in L_{\mathrm{loc}}^{p}(\R^{3})$ with $p\in [1,\frac{4}{3}]$ and a sequence of divergence-free vector fields $b_n\in C^\infty(\R^{3})$ converging to $b$ strongly in $L^p_{\mathrm{loc}}(\R^3)$ such that the following holds. Let $u_0\in \mathcal{I}\cap L^\infty(\R^3)$ be a given inital datum, then there exist subsequences $n_i$ and $n_j$ such that the sequences $u_{n_i}$ and $u_{n_j}$, solutions of \eqref{eq:tee} with initial datum $u_0$, converge in $L^\infty((0,T);L^\infty(\R^3)-w*)\cap L^\infty((0,T);L^1_{\mathrm{loc}}(\R^3))$ to two different limits, which are bounded distributional solutions of \eqref{eq:te}.
\end{thm} 
The above theorem is a consequence of the following analogous result for the flow:
\begin{thm}\label{teo:main2}
There exists a divergence-free vector field $b\in L_{\mathrm{loc}}^{p}(\R^{3})$ with $p\in [1,\frac{4}{3}]$ and a sequence of divergence-free vector fields $b_n\in C^\infty(\R^{3})$  such that $b_n\to b$ strongly in $L_{\mathrm{loc}}^{p}(\R^{3})$ and the uniquely defined sequence $X^n$ of flows of $b_n$ does not converge, but has at least two different subsequences converging in $L^\infty((0,T);L^1_{\mathrm{loc}}(\R^3))$ to two different flows of $b$.
\end{thm}
For several PDEs, selection principles or admissibility criteria are needed when the regularity of weak solutions is not enough to guarantee uniqueness. For example, this is the case for scalar conservation laws: if we consider weak solutions satisfying in addition the entropy inequality it is possible to prove uniqueness. In the context of the incompressible Euler equations general admissibility criteria, that can be satisfied by only one weak solution, are not known when the initial datum $u_0\in L^2$. Contrary to the case of scalar conservation laws, criteria based on an energy inequality are known not to select a unique solution, as proved in \cite{DLS}. Another natural approach would be to consider weak solutions of Euler equations obtained as limit of Navier-Stokes equations. In this regard, in \cite{BTW} the authors prove that for shear-flow solutions of the Euler equations, the vanishing viscosity limit of Leray weak solutions of the Navier-Stokes equations selects a unique solution. On the other hand, the recent result in \cite{BV} shows that the limit of weak solutions of Navier-Stokes equations, which are not Leray weak solutions, does not select a unique solution. Therefore it is fair to say that there is not a clear picture of selection principles in fluid dynamics. Our result shows that, already for the linear transport equation, the very natural approximation procedure of smoothing the vector field does not select a unique solution.\par
It is worth pointing out that, differently from the nonuniqueness examples obtained via convex integration, the approximation constructed here is explicit and consists of functions $u^\e$ which are the unique exact solutions of \eqref{eq:tee}.\par In this spirit, the problem of selection for bounded solutions can also be posed for other types of approximations which guarantee uniqueness at the approximate level, such as
\begin{enumerate}
\item[(Q2)] \emph{Does the approximation procedure obtained by smoothing the vector field via a convolution with a suitably chosen mollifier select a unique solution of \eqref{eq:te}?} 
\vspace{0.3cm}
\item[(Q3)] \emph{Does the approximation procedure obtained by vanishing viscosity limit of}  
\begin{equation}\label{eq:tde}
\partial_t u^\e+b\cdot\nabla u^\e=\e\Delta u^\e,
\end{equation}
\emph{select a unique solution of \eqref{eq:te}?}
\end{enumerate}
Unfortunately we are not able to provide an answer to the two questions above with the techniques of this work. Nevertheless if one looks to a slightly different version of (Q3), considering $u^\e$ as the solution of
\begin{equation}\label{eq:tder}
\partial_t u^\e+b_{\e}\cdot\nabla u^\e=\e\Delta u^\e,
\end{equation}
in which we also regularize the vector field, an easy corollary of our main theorem exploiting a diagonal argument shows that there exists a vector field $b$ and a smooth approximation $b_{\e}$ for which the selection of a unique solution as limit of solutions of \eqref{eq:tder} does not hold, see Corollary \ref{cor:visc}.\\
We point out that in \cite{AF} the authors study the behaviour of solutions of \eqref{eq:te} in dimension one, constructed as limit of two different approximations: (i) $u^\e_{visc}$ solution of \eqref{eq:tde}; (ii) $u^\e_{stoc}$ solution of \eqref{eq:te} with a multiplicative noise of the form  $\e\nabla u\circ\frac{\de W(t)}{\de t}$ at the right hand side, where $W(t)$ is a one-dimensional Brownian motion. In particular, they show that in the limit the sequences $u^\e_{visc}$ and $u^\e_{stoc}$ converge to two different solutions of the limit equation.
\bigskip

\subsection*{Organization of the paper}
The paper is organized as follows. In Section~2 we recall some of the main notions and results that will be exploited in the sequel. In Section 3 we define the limit vector field, we introduce the regularizing sequence of vector fields, and we prove some of their main properties. Finally in Section 4 we prove our main results.

\section{Preliminaries and background}
We start by recalling some basic definitions. 
\begin{defn}[Distributional solutions]
Let $b\in L^1_{\mathrm{loc}}((0,T);L^1_{\mathrm{loc}}(\R^d;\R^d))$ be divergence-free and $u_0\in L^\infty_{\mathrm{loc}}(\R^d)$ be given. A function $u$ is called a distributional solution of ($\ref{eq:te}$) if $u\in L^\infty((0,T);L^\infty_{\mathrm{loc}}(\R^d))$ and

\smallskip

\[
\iint u(\partial_t\varphi+b\cdot\nabla\varphi )\de x\de t+\int u_0\varphi_{|_{t=0}} \de x=0,
\]
\smallskip
for any $\varphi\in C^\infty_c([0,T)\times\R^d)$.
\end{defn}
A very general existence theorem for weak solutions can be proved along the lines sketched in the introduction; we refer to \cite{DPL} for a detailed proof.

\smallskip

\begin{thm}
Let $b\in L^1_{\mathrm{loc}}((0,T);L^1_{\mathrm{loc}}(\R^d;\R^d)$ be divergence-free and $u_0\in L^\infty(\R^d)$. There exists a weak solution $u\in L^\infty([0,T];L^\infty(\R^d))$ of ($\ref{eq:te}$).
\end{thm}
%
%
%
%

Next, we recall some results regarding the uniqueness of solutions of the transport equation \eqref{eq:te} and the associated ordinary differential equations \eqref{eq:ode}. We start with the notion of {\em regular Lagrangian flow}, introduced by Ambrosio in \cite{Am},
\begin{defn}
Let $b\in L^1((0,T);L^1_{\mathrm{loc}}(\R^d;\R^d))$ be given. We say that $X:(0,T)\times\R^d\rightarrow\R^d$ is a regular Lagrangian flow associated to $b$ if
\begin{enumerate}
\item for a.e. $x\in\R^d$ the map $t\mapsto X(t,x)$ is an absolutely continuous integral solution of the ordinary differential equation
\begin{equation}
\begin{dcases}
\frac{\de}{\de t}X(t,x)=b(t,X(t,x)), \\
X(0,x)=x,
\end{dcases}
\end{equation}
\item there exists a constant $L$ indipendent of $t$ such that
\begin{equation}\label{eq:incom}
X(t,\cdot)\# \mathscr{L}^d\leq L \mathscr{L}^d.
\end{equation}
\end{enumerate}
\end{defn}
With the notation $X(t,\cdot)\# \mathscr{L}^d$ we denote the push-forward measure of $\mathscr{L}^d$ through $X$ which is defined as
$$
X(t,\cdot)\#\mathscr{L}^d(A)=\mathscr{L}^d(X(t,\cdot)^{-1}(A)),\hspace{0.7cm}\mbox{for all Borel sets }A\subseteq\R^d.
$$
In the case of a divergence-free vector field, $L$ can be taken to be $1$ and \eqref{eq:incom} is an equality.
As already stressed in the introduction, in order to prove uniqueness of solutions more information on the regularity and on the growth of the vector field is needed. We recall the following theorem, proved in \cite{Am}:
\begin{thm}{\label{teo:am}}
Let $b\in L^1((0,T);BV_{\mathrm{loc}}(\R^d;\R^d))$ be a vector field satisfying $\dive b\in L^1((0,T);L^\infty(\R^d))$ and the growth condition
\[
\frac{|b(t, x)|}{1 + |x|} \in L^1((0,T);L^1(\R^d))+L^1((0, T);L^\infty(\R^d)).
\]
Then there exist:
\begin{itemize}
\item a unique bounded distributional solution of \eqref{eq:te};
\item a unique regular Lagrangian flow $X$ of $b$.
\end{itemize}
\end{thm}
For an alternative approach, based only on {\em a priori} estimates on the flow, we refer to \cite{CDL} for $W^{1,p}(\R^d)$ vector fields with $p>1$ and to \cite{CB},\cite{CNSS} for the case $p=1$ and vector fields the gradient of which is a singular integral of a function in $L^1(\R^d)$. This latter is a class of interest in the context of the $2D$ Euler equations. More recently, these results were improved in \cite{N} to vector fields which can be represented as singular integral of a function in $BV(\R^d)$. We conclude this section recalling the following stability theorem from \cite{N}.
\begin{thm}\label{teo:stab}
Let $b_n$ be a sequence of smooth vector fields converging in $L^1((0,T);L^1_{\mathrm{loc}}(\R^d))$ to a vector field $b\in L^1((0,T);BV_{\mathrm{loc}}(\R^d;\R^d))$, with $\dive b\in L^1((0,T);L^\infty(\R^d))$ and satisfying the growth condition
\[
\frac{|b(t, x)|}{1 + |x|} \in L^1((0,T);L^1(\R^d))+L^1((0, T);L^\infty(\R^d)).
\]
Assume that for some decomposition 
$$ 
\frac{|b_n(t, x)|}{1 + |x|}=\tilde{b}_{n,1}(t,x)+\tilde{b}_{n,2}(t,x)$$ we have
$$
\| \tilde{b}_{n,1} \|_{L^1((0,T);L^1(\R^d))} + \| \tilde{b}_{n,2} \|_{L^1((0, T);L^\infty(\R^d))} \leq C
$$
for all $n\in \N$ and for some constant $C$. Then the following statements hold true.
\begin{itemize}
\item Let $X_n$ and $X$ be the regular Lagrangian flows associated respectively to $b_n$ and $b$, denote with $L_n$ and $L$ the compressibility constants of the flows and assume that the sequence $L_n$ is bounded uniformly in $n$. Then, $X_n\to X$ locally in measure uniformly in time; that is, for every compact set $K\subset \R^d$
$$
\sup_{[0,T]}\int_K 1\wedge |X_n(t,x)-X(t,x)|\de x \to 0 \hspace{0.5cm} \mathrm{as} \ n\to +\infty.
$$
\item Let $u_0\in L^\infty(\R^d)$ and let $u_n$ be the weak solution of the Cauchy problem $(\ref{eq:te})$ with initial datum $u_0$ and vector field $b_n$. Then, $$u_n\to u \ in \ L^\infty((0,T);L^\infty(\R^d)-w*)\cap L^\infty((0,T);L^1_{\mathrm{loc}}(\R^d))$$ where $u$ is the unique solution of $(\ref{eq:te})$ with initial datum $u_0$ and vector field $b$.
\end{itemize}
\end{thm}

\par

\section{The vector field}
In this section we introduce the vector field $b$, which will be the limit of our approximation as stated in Theorems \ref{teo:main1} and \ref{teo:main2}. More precisely, we look for a vector field for which uniqueness of the flow fails. 

\subsection{A 2D example of DiPerna and Lions.} 
It is worth recalling the following example due to DiPerna and Lions \cite{DPL}.
\begin{exm}\label{eq:dpl}
Define the two dimensional vector field $b=(b_1,b_2)$ as
\begin{equation}\label{vdpl}
\begin{dcases}
b_1(x,y)= -\mathrm{sgn}(y) \left(\frac{x}{|y|^2} \chi_{ \{ |x|\leq |y|  \}}+ \chi_{ \{ |x|> |y|  \}}  \right), \\
b_2(x,y)=-\left(\frac{1}{|y|}\chi_{ \{ |x|\leq |y|  \}}+\chi_{ \{ |x|> |y|  \}} \right).
\end{dcases}
\end{equation}
The vector field $b\in W^{s,1}_{\mathrm{loc}}(\R^2;\R^2)$ for all $s\in[0,1)$, $\dive b=0$ in the sense of distributions, $b\in L^p+L^\infty$ for all $p\in [1,2)$.
We can define two different regular Lagrangian flows of $b$ that preserve the Lebesgue measure. In particular, they are different on the set $\{ (x,y)\in\R^2:0<x<y \}$ and are defined as follows 
\[
\begin{cases}
X_1(t,x,y)=\frac{x}{y}\sqrt{|y^2-2t|}, \\
X_2(t,y)=\sigma\sqrt{|y^2-2t|},
\end{cases}
\]
and
\[
\begin{cases}
\tilde{X}_1(t,x,y)=\sigma\frac{x}{y}\sqrt{|y^2-2t|}, \\
\tilde{X}_2(t,y)=\sigma\sqrt{|y^2-2t|},
\end{cases}
\]
where $\sigma=1$ if $t\leq y^2/2$ and $\sigma=-1$ if $t>y^2/2$.\\
\begin{figure}
\centering
\includegraphics[width=0.75\textwidth]{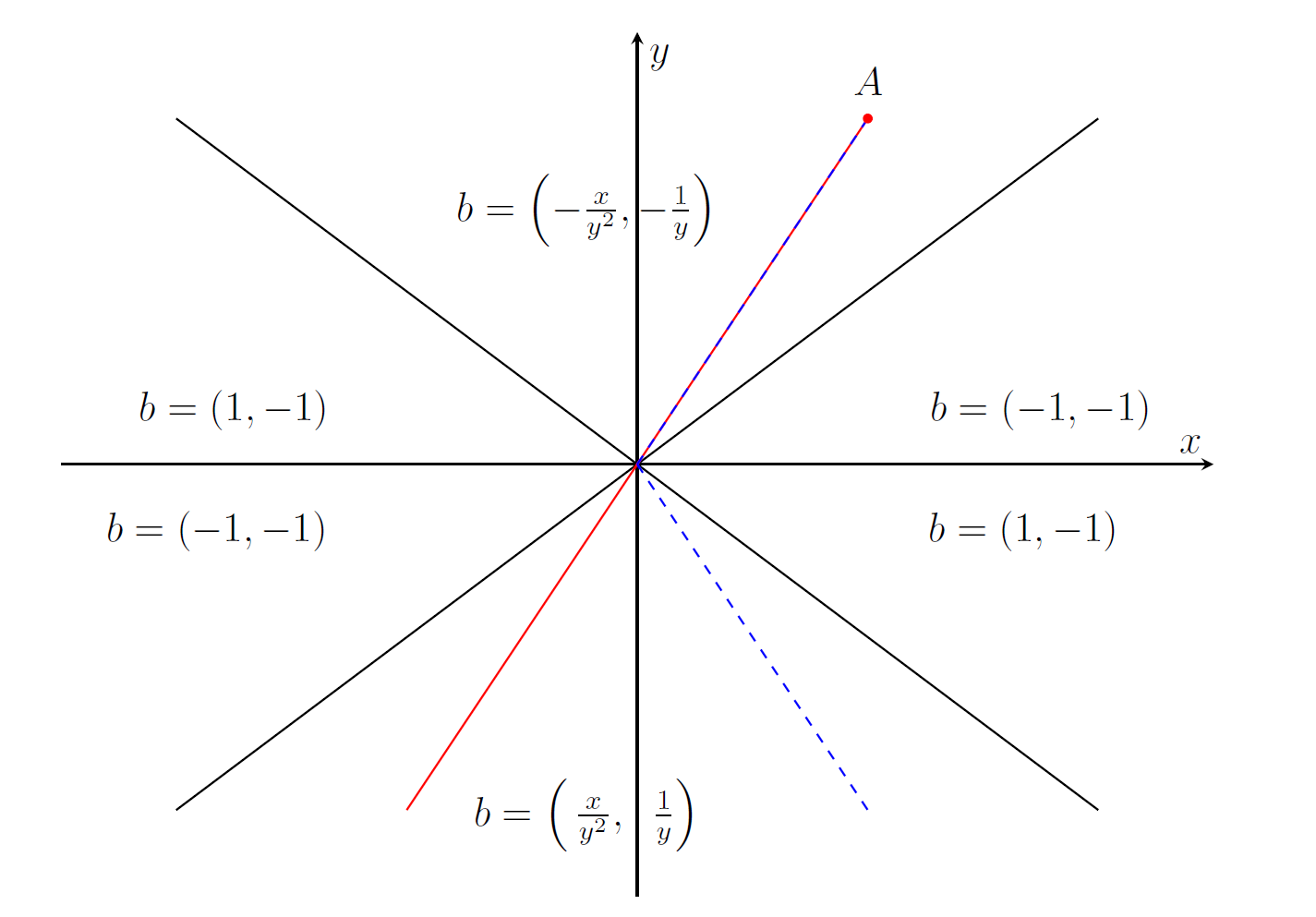}
\caption{The blue line is a characteristic of the flow $X$ while the red line is a characteristic of the flow $\tilde{X}$.}
\end{figure}
\begin{figure}
\centering
\includegraphics[width=0.75\textwidth]{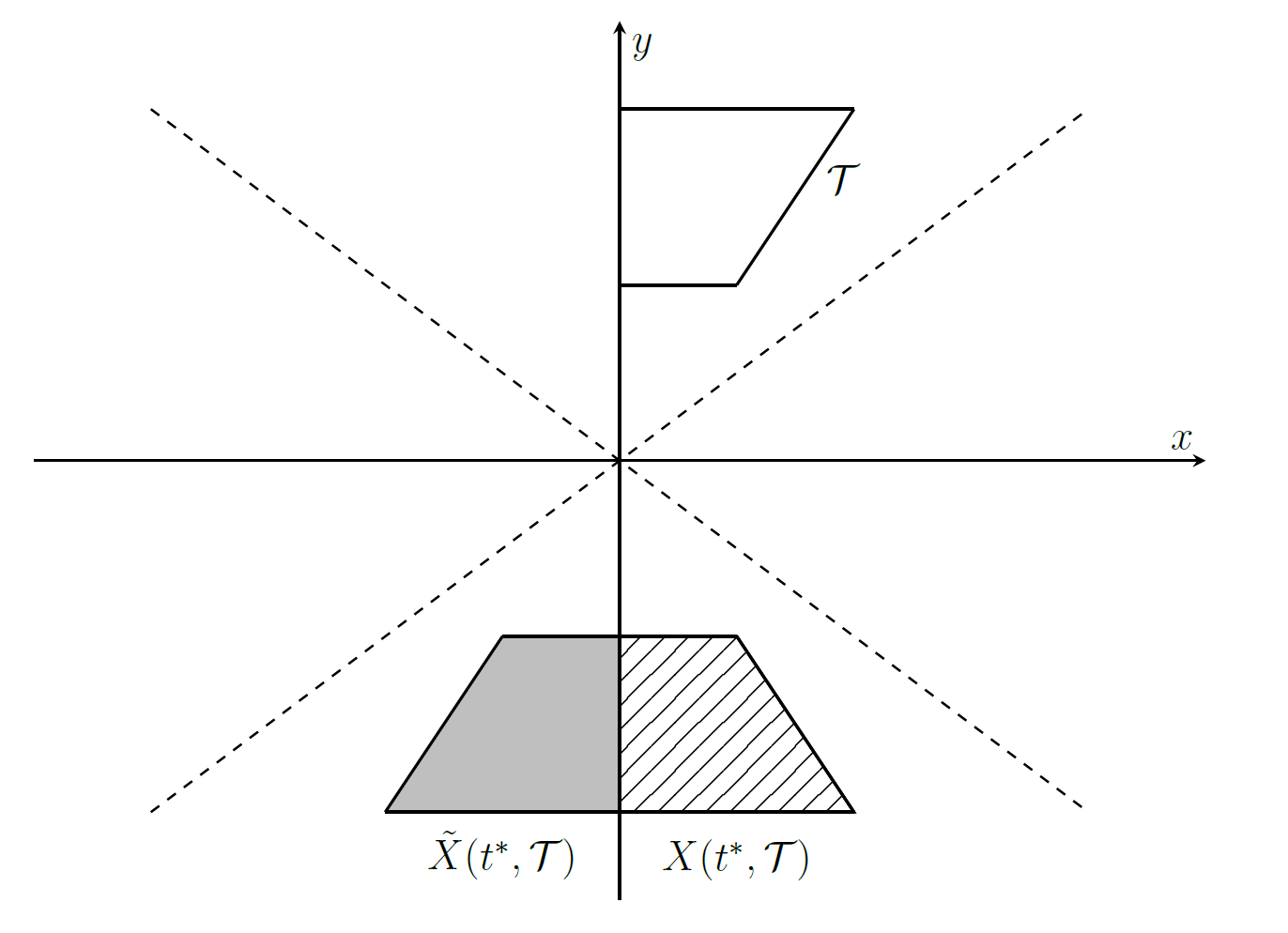}
\caption{Action of the flows on the trapezium $\mathcal{T}$.}
\end{figure}
The nonuniqueness of the flow has the following geometric interpetration: consider the trapezium $\mathcal{T}$ in the half plane $\{ y>0 \}$ as in Figure 2, then there exists a time $t^*$ such that 
\begin{itemize}
\item the region filled with diagonal lines is $X(t^*,\mathcal{T})$ and it is symmetric to $\mathcal{T}$ with respect to $\{y=0\}$;
\item the grey region is $\tilde{X}(t^*,\mathcal{T})$ and it is symmetric to $\mathcal{T}$ with respect to $(0,0)$.
\end{itemize}
\end{exm}
We would like to use this example to give a negative answer to (Q1). It is not a problem to construct a smooth approximation of \eqref{vdpl} which gives $X$ in the limit. Instead, it is not clear to us how to get $\tilde{X}$ in the limit: we are not able to construct an approximation $b_\e$ of \eqref{vdpl} avoiding intersections of trajectories for fixed $\e$. In order to avoid this topological problem, we rather work in three space dimensions.
\subsection{The limit vector field}
We now introduce the vector field
\begin{equation}
b(x,y,z)=
\begin{dcases}
\left(- \sgn(z) \frac{x}{|z|^2},- \sgn(z) \frac{y}{|z|^2}, - \frac{2}{|z|} \right)  & \mathrm{if} \ x\in P,\\ 
(0,0,0) & \mathrm{otherwise},
\end{dcases}
\label{eq:b}
\end{equation}
where $P\subset\R^{3}$ denotes the set 
\[
P=P^+\cup P^-=\{ (x,y,z)\in\mathbb{R}^3: x^2+y^2\leq z \}\cup \{\ (x,y,z)\in\mathbb{R}^3:x^2+y^2\leq -z\},
\]
being the union of two symmetric paraboloids.
\begin{figure}[h!]
\centering
\includegraphics[width=0.6\textwidth]{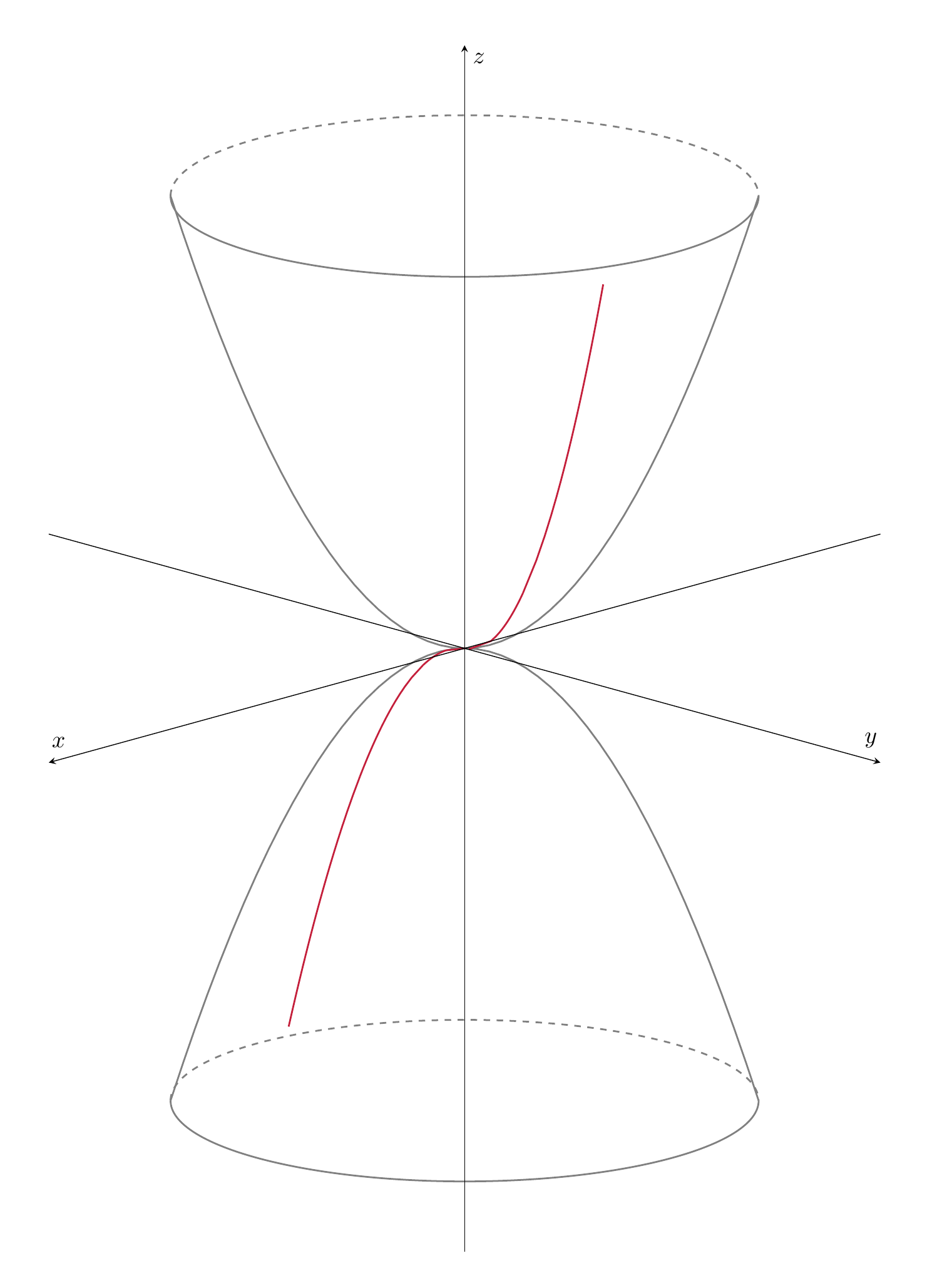}
\caption{An example of flow $X^\Theta$}
\end{figure}
The vector field $b\in L^p_{\mathrm{loc}}(\R^3)$ for all $p\in \left[1,\frac{4}{3} \right]$ and it can be directly checked that $\dive b=0$ in the sense of distributions on the whole $\R^{3}$, in particular $b$ is tangent to $\partial P$. Moreover the vector field $b$ satisfies the growth conditions of Theorem \ref{teo:am}. Observe that this vector field does not belong to any Sobolev space $W^{1,p}(\R^3)$ or to $BV(\R^3)$. However, similarly to the vector field in Example \ref{eq:dpl}, we have that $b\in W^{s,1}_{\mathrm{loc}}(\R^3)$ for $0<s<1/2$.
\\
We can easily define infinitely many different regular Lagrangian flows of $b$. Since we are considering flows defined almost everywhere, we need to define them only on $\R^3\setminus \{0\}$. We start for $\mathbf{x}\in\R^3\setminus P$: in this region the vector field is identically $0$ so that we define a flow $X$ simply as
$$
X(t,\mathbf{x})=\mathbf{x} \hspace{1cm}\forall t\geq 0.
$$
If $\mathbf{x}=(x,y,z)\in P^-$ we define
\begin{equation}\label{x}
\begin{dcases}
X_1(t,x,z)=\frac{x}{\sqrt{-z}}\sqrt[4]{z^2+4t} \\
X_2(t,y,z)=\frac{y}{\sqrt{-z}}\sqrt[4]{z^2+4t} \\
X_3(t,z)=-\sqrt{z^2+4t}
\end{dcases}
\hspace{0.7cm} \forall \ t\geq 0.
\end{equation}
Finally, when $\mathbf{x}=(x,y,z)\in P^+$ define the flow as
\begin{equation}\label{x}
\begin{dcases}
X_1(t,x,z)=\frac{x}{\sqrt{z}}\sqrt[4]{z^2-4t} \\
X_2(t,y,z)=\frac{y}{\sqrt{z}}\sqrt[4]{z^2-4t} \\
X_3(t,z)=\sqrt{z^2-4t}
\end{dcases}
\hspace{0.7cm} \mathrm{for} \ t \in \left[0,\frac{z^2}{4} \right].
\end{equation}
At time $t=\frac{z^2}{4}$ the trajectories reach the origin. A formal computation shows that the quantity
$$
\frac{X_1^2(t,\mathbf{x})+X_2^2(t,\mathbf{x})}{|X_3(t,\mathbf{x})|}=\frac{x^2+y^2}{|z|}
$$
is conserved by solutions of (3.2). This suggest to define the flow as
\begin{equation}\label{eq:thetasol}
\begin{dcases}
X_1(t,x,z)=\frac{x}{\sqrt{z}}\sqrt[4]{4t-z^2} \cos\Theta - \frac{y}{\sqrt{z}}\sqrt[4]{4t-z^2} \sin\Theta\\
X_2(t,y,z)=\frac{x}{\sqrt{z}}\sqrt[4]{4t-z^2} \sin\Theta+\frac{y}{\sqrt{z}}\sqrt[4]{4t-z^2} \cos\Theta \\
X_3(t,z)=-\sqrt{4t-z^2}
\end{dcases}
\hspace{0.7cm} \forall t \geq\frac{z^2}{4}.
\end{equation}
where $\Theta\in (0,2\pi]$ is arbitrary. An easy computation shows that $X$, defined as above, is a regular Lagrangian flow of $b$ for every $\Theta\in (0,2\pi]$. We call those kind of solutions $X^\Theta$, where $\Theta$ represents a rotation in the $xy$ plane. Heuristically, we can define this kind of flows as a consequence of the fact that the trajectories, once they reach the origin, can come out arbitrarily. The lack of uniqueness is a consequence of the fact that all the solutions can be extended in infinitely many ways once they reach the origin. This reproduces the same mechanism of Example $\ref{eq:dpl}$, although in this case the additional dimension allows for more flexibility and for an easy and explicit description of the nonunique flows. Actually there are other possible ways to define regular Lagrangian flows of $b$. This is not important for the purpose of this work and we refer to \cite{CCS} for a more in-depth discussion.

\subsection{The approximation of the limit vector field}
In this section we provide an approximation $b_\e$ of the vector field $b$ such that, for a fixed $\Theta\in(0,2\pi]$, the sequence $X^\e$ of flows of $b_\e$ converges to $X^\Theta$. Our strategy is to approximate the vector field $b$ close to the origin forcing the trajectories to rotate very fast along a given helix. In order to do this, we first smooth the union of the two paraboloids in the origin, see Figure 4. Then, we choose the rotation velocity in the cylinder $C_\e$ to be proportional to the height of $C_\e$. Precisely, the smaller the height of the cylinder, the faster the rotation of the characteristics. In order to get a smooth transition for the vector field between the truncated paraboloids $P^+_\e,P^-_\e$ and the cylinder, we then consider two transition zones $T^+_\e, T^-_\e$ (see again Figure 4). Finally, we define the region $\tilde{P}^\e$ as 
$$
\tilde{P}^\e=P^+_\e\cup T^+_\e\cup C_\e\cup T^-_\e\cup P^-_\e.
$$
The main properties of the sequence of approximating vector fields $\{b_{\e}\}_{\e}$ that we will construct are described in the following proposition.
\begin{prop}\label{prop:main}
Let $b$ be the vector field in $\eqref{eq:b}$. Given $\Theta\in (0,2\pi]$ there exists a sequence of vector fields $b_\e$ such that
\begin{enumerate}
\item $b_\e$ converges to $b$ in $L^1_{\mathrm{loc}}(\R^3)$;
\item $\dive b_\e=0$ in the sense of distributions, in particular $b_\e$ is tangent to $\partial \tilde{P^\e}$;
\item the flow $X^\e$ of $b_\e$ converges uniformly to $X^\Theta$ and the same convergence holds for the inverse $\left(X^\e\right)^{-1}$ towards $\left(X^\Theta\right)^{-1}$;
\item $b_\e\in \mathrm{Lip}(\tilde{P}^\e)$, $b_\e$ is identically $0$ on $\R^3\setminus\tilde{P}^\e$, and $b_\e\in BV_{\mathrm{loc}}(\R^3)$;
\item $\frac{b_\e}{1+|x|}=b_{1,\e}+b_{2,\e}$, with $b_{1,\e}\in L^1(\R^3)$ and $b_{2,\e}\in L^\infty(\R^3)$.
\end{enumerate}
\end{prop}
\begin{proof} We divide the proof in the following steps.\\
\\
$\underline{\mathit{Step \ 1}}$\\
\\
For any $\e>0$ we define:
\begin{equation}\label{eq:be}
b_\e(x,y,z)=
\begin{cases}
\left( -\frac{x}{|z|^2},-\frac{y}{|z|^2}, - \frac{2}{|z|} \right) & \mathrm{in} \ P^+_\e,
\\
\left( b_1(x,y,z), b_2(x,y,z), b_3(z) \right) & \mathrm{in} \ T^+_\e,
\\
\left(-\frac{y}{\beta^2\varepsilon^2},\frac{x}{\beta^2\varepsilon^2}, -\frac{27}{16\beta\varepsilon} \right) & \mathrm{in} \ C_\e,
\\
\left(\bar{b}_1(x,y,z), \bar{b}_2(x,y,z), \bar{b}_3(z) \right) & \mathrm{in} \ T^-_\e,
\\
\left( \frac{x}{|z|^2},\frac{y}{|z|^2}, - \frac{2}{|z|} \right) & \mathrm{in} \ P^-_\e,
\\
(0,0,0) & \mathrm{otherwise}.
\end{cases}
\end{equation}
In the above formula $\left( b_1(x,y,z), b_2(x,y,z), b_3(z) \right)$ and $\left(\bar{b}_1(x,y,z), \bar{b}_2(x,y,z), \bar{b}_3(z) \right)$ will be defined in the following, while $\alpha,\beta,\gamma,\eta \in \R_+$ and
\begin{equation*}
\begin{aligned}
&P^+_\e:=\{ (x,y,z)\in\R^3: x^2+y^2\leq z, \ z\geq\alpha\e \},\\
&T^+_\e:=\left\{ (x,y,z)\in\R^3: \beta\e+\beta\e\sqrt{\frac{27(x^2+y^2)-32\beta\e}{27(x^2+y^2)}}\leq z, \ z\in [\beta\e,\alpha\e]  \right\},\\
&C_\e:=\left\{ (x,y,z)\in\R^3: x^2+y^2\leq \frac{32}{27}\beta\e, z\in [-\gamma\e, \beta\e] \right\},\\
&T^-_\e:=\left\{ (x,y,z)\in\R^3: -\gamma\e-\gamma\e\sqrt{\frac{27(x^2+y^2)-32\gamma\varepsilon}{27(x^2+y^2)}}\leq -z, \ z\in [-\eta\e,-\gamma\e]  \right\},\\
&P^-_\e:=\{ (x,y,z)\in\R^3: x^2+y^2\leq -z, \ z\leq -\eta\e \}.
\end{aligned}
\end{equation*}
In the regions $T_\e^+$ and $T_\e^-$, here referred to as {\em transition zones}, we combine the effects of rotation and dilation for the first two components. This means that in $T^+_\e$ we define $b_1$ and $b_2$ by interpolating linearly in the $z$ variable the value of $b_\e(x,y,\alpha\e)$ and $b_\e(x,y,\beta\e)$. The third component and the geometry of the regions are defined in order to have that $b_\e$ is a divergence-free vector field: this means that we obtain $b_3$ and the definition of $T^+_\e$ by solving the equations
$$
\partial_x b_1+\partial_y b_2+\partial_z b_3=0,
$$
$$
{b_\e}|_{\partial T^+_\e}\cdot n=0,
$$
where $n$ denotes the unit exterior normal to $\partial T^+_\e$. With the same idea we define $T^-_\e,\bar{b}_1,\bar{b}_2,\bar{b}_3$. So for $(x,y,z)\in T_\e^+$, the vector field $b_{\e}$ is defined as:
$$
b_1(x,y,z)=\frac{z-\beta \varepsilon}{\varepsilon} \frac{x}{\alpha^2\varepsilon^2(\beta-\alpha)}-\frac{z-\alpha\varepsilon}{\varepsilon}\frac{y}{\beta^2\varepsilon^2(\beta-\alpha)},
$$

$$
b_2(x,y,z)=\frac{z-\alpha \varepsilon}{\varepsilon} \frac{x}{\beta^2\varepsilon^2(\beta-\alpha)}+\frac{z-\beta\varepsilon}{\varepsilon}\frac{y}{\alpha^2\varepsilon^2(\beta-\alpha)},
$$
$$
b_3(z)=\frac{2}{\alpha^2\varepsilon^3(\beta-\alpha)} \left(\beta\varepsilon z-\frac{z^2}{2}\right).
$$
\begin{figure}[h!]\label{fig:clessidra}
\centering
\includegraphics[width=0.75\textwidth]{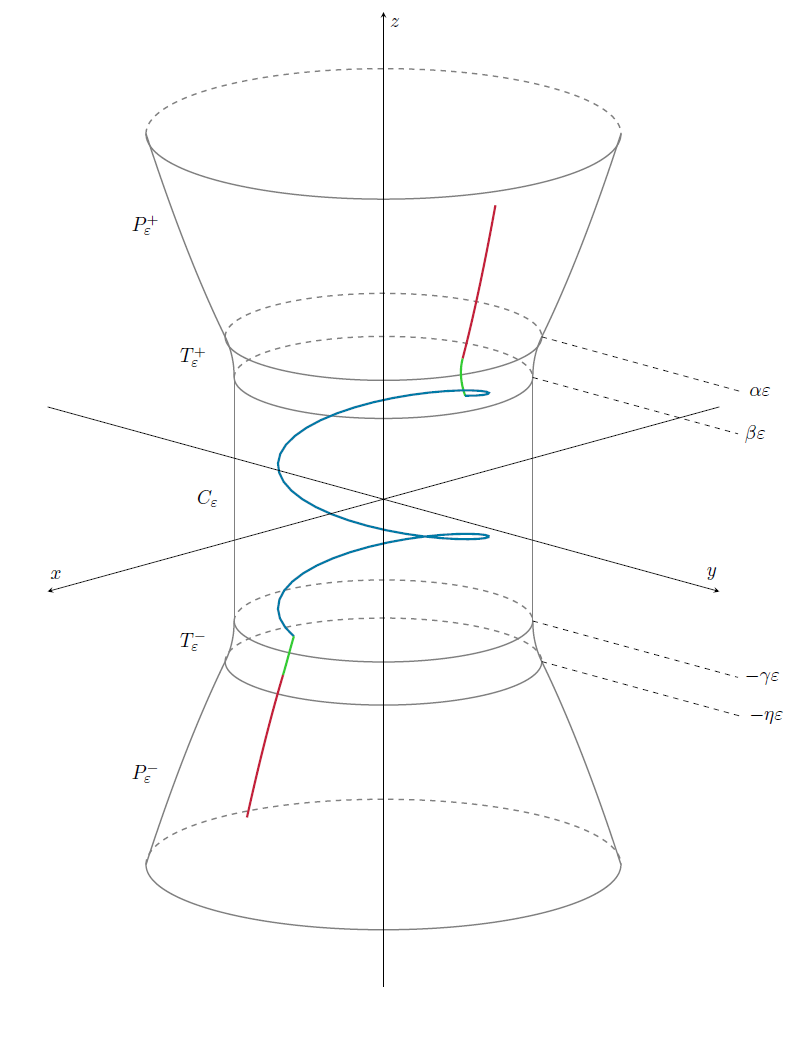}
\caption{The flow $X^\e$ is represented by different colors according to the region in which it is located. In the limit, it converges to the one of Figure 3.}
\end{figure}\\
Instead, for $(x,y,z)\in T_\e^-$, the vector field $b_{\e}$ is defined as:
\[
\bar{b}_1(x,y,z)=\frac{z+\gamma \varepsilon}{\varepsilon} \frac{x}{\eta^2\varepsilon^2(\gamma-\eta)}+\frac{z+\eta\varepsilon}{\varepsilon}\frac{y}{\beta^2\varepsilon^2(\gamma-\eta)},
\]

\[
\bar{b}_2(x,y,z)=-\frac{z+\eta\varepsilon}{\varepsilon}\frac{x}{\beta^2\varepsilon^2(\gamma-\eta)}+\frac{z+\gamma \varepsilon}{\varepsilon} \frac{y}{\eta^2\varepsilon^2(\gamma-\eta)},
\]
\[
\bar{b}_3(z)=-\frac{2}{\eta^2\varepsilon^3(\gamma-\eta)}\left(\frac{z^2}{2}+\gamma\varepsilon z \right) .
\]
Moreover, in order to connect the various regions, the parameters are chosen so that:
\[
4\beta=3\alpha, \ 4\gamma=3\eta, \ \beta=\gamma.
\]
We remark that $\beta$ is the only free parameter, representing the half height of the cylinder, and it will be chosen later in the proof.
The vector fields $b_\varepsilon$ and $b$ differ only in the set $A_\e:=T_\e^+ \cup C_\e \cup T_\e^-$. Since $\mathscr{L}^3(A_\e)=C\e^2$, $\|b_\e\|_\infty\leq C\e^{-3/2}$ and $b\in L^1_{\mathrm{loc}}(\R^3)$, by the trivial estimate
\[
\int_{A_\varepsilon} |b_\varepsilon -b| \ \de x  \leq \int_{A_\varepsilon} |b_\varepsilon|\de x +\int_{A_\varepsilon}|b| \de x
\]
we get the $L^1_{\mathrm{loc}}$ convergence of $b_\e$ to $b$.\\
\\
$\underline{\mathit{Step \ 2}}$\\
\\
We now compute the characteristics of the vector field $b_{\e}$ for $\mathbf{x}\in P^+_\e$, as it is the region of interest for the nonuniqueness. Similar computations  allow to compute the characteristics in the whole $\R^3$ and so we omit them. Consider the following system of ordinary differential equations
\begin{equation}\label{abc}
\begin{cases}
\dot{X^\varepsilon}(t,\mathbf{x})=b_\varepsilon(X^\varepsilon(t,\mathbf{x})) \\
X^\varepsilon(0,\mathbf{x})=\mathbf{x}
\end{cases} \ \mathbf{x}\in P^+_\e.
\end{equation}
Since $b_\e$ is smooth on $P^+_\e$, \eqref{abc} has a unique solution given by:
\[
\begin{cases}
X_1^\varepsilon(t,x,z)=\frac{x}{\sqrt{z}}\sqrt[4]{z^2-4t} \\
X_2^\varepsilon(t,y,z)=\frac{y}{\sqrt{z}}\sqrt[4]{z^2-4t} \\
X_3^\varepsilon(t,z)=\sqrt{z^2-4t}
\end{cases}
\ t \in \left[0,t_1^\varepsilon:=\frac{z^2-\alpha^2\varepsilon^2}{4} \right].
\]
At $t=t_1^\varepsilon$, we have $X_3^\varepsilon(t_1^\varepsilon,z)=\alpha\varepsilon$ and the equations change. Specifically we have for the third component
\[
\begin{cases}\label{x3}
\dot{X_3^\e}(t,z)=-\frac{27}{8\beta^3\e^3}\left( \beta\e X_3^\e-\frac{(X_3^\e)^2}{2} \right), \\
X_3^\e(t_1^\e,z)=\alpha\e.
\end{cases}
\]
The solution is
\begin{equation}\label{eqz}
\z=\frac{4\beta\e}{2+\exp\left(\frac{27}{8\beta^2\e^2}(t-t_1^\e)\right)},
\end{equation}
up to the time $t_2^\e:=t_1^\e+\frac{8\beta^2\e^2}{27}\ln(2)$, when $\z(t_2^\e,z)=\beta\e$.
Substituting \eqref{eqz} in the first two equations, we can rewrite them  in the form
\begin{equation}\label{eqtr}
\begin{cases}
\dot{\x}=A(t)\x-B(t)\y, \\ 
\dot{\y}=B(t)\x+A(t)\y, \\
\x(t_1^\e,x,z)=\frac{x}{\sqrt{z}}\sqrt{\alpha\e}, \\
\y(t_1^\e,y,z)=\frac{y}{\sqrt{z}}\sqrt{\alpha\e},
\end{cases}
\end{equation}
where the coefficients $A(t), B(t)$ are defined as
$$
A(t)=-\frac{27}{16}\frac{\z-\beta\e}{\beta^3\e^3},
$$
$$
B(t)=-\frac{3\z-4\beta\e}{\beta^3\e^3}.
$$
Multiplying the first equation by $\x$ and the second one by $\y$,  adding the two equations and setting $\varphi_\e=(\x)^2+(\y)^2$, we obtain that
\[
\begin{cases}
\dot{\varphi_\e}=2 A(t)\varphi_\e, \\
\varphi_\e(t_1^\e)=\frac{x^2+y^2}{z}\alpha\e,
\end{cases}
\]
yielding
\begin{equation}
\varphi_\e(t)=\frac{x^2+y^2}{z}\alpha\e e^{-\frac{27}{8\beta^2\e^2}(t-t_1^\e)} \left( \frac{2+\exp\left( \frac{27}{8\beta^2\e^2}(t-t_1^\e) \right)}{3} \right)^2.
\end{equation}
Because $\x=\sqrt{\varphi_\e} \cos\theta$ and $\y=\sqrt{\varphi_\e}\sin\theta$, substituting these expressions in the equations \eqref{eqtr} we get
\[
\dot{\theta}(t)=B(t),
\]
and then
\begin{equation}
\theta(t)=\theta_0+\frac{2}{\beta^2\e^2}(t-t_\e^1)-\frac{16}{9}\ln\left(\frac{2+\exp\left( \frac{27}{8\beta^2\e^2}(t-t_1^\e) \right)}{3} \right),
\end{equation}
where 
$$
\theta_0=\begin{cases}
\pi \hspace{4cm} &\mbox{if } y=0,x<0, \\
2\arctan\left(\frac{y}{x+\sqrt{x^2+y^2}}\right) & \mbox{otherwise}.
\end{cases}
$$
During the passage in the first transition zone, the trajectory rotates by an angle 
\[
\bar{\theta}=\theta(t_2^\e)-\theta_0=\frac{16}{27}\ln(2)+\frac{16}{9}\ln\left(\frac{3}{4}\right).
\]
At time $t_2^\e$ the flow enters the cylinder and the system becomes
\[
\begin{cases}
\dot{\x}=-\frac{\y}{\e^2},\\
\dot{\y}=\frac{\x}{\e^2},\\
\dot{\z}=-\frac{27}{16\beta\e}.
\end{cases}
\]
Then, the solution can be extended as
\begin{equation}
\begin{cases}
\x(t)=\x(t_2^\e)\cos\left( \frac{t-t_2^\e}{\e^2} \right)-\y(t_2^\e)\sin\left( \frac{t-t_2^\e}{\e^2} \right),\\
\y(t)=\x(t_2^\e)\sin\left( \frac{t-t_2^\e}{\e^2} \right)+\y(t_2^\e)\cos\left( \frac{t-t_2^\e}{\e^2} \right),\\
\z(t)=\beta\e-\frac{27}{16\beta\e} (t-t_2^\e),
\end{cases}
\end{equation}
up to the time $t_3^\e:=t_2^\e+\frac{32}{27}\beta^2\e^2$ when $\z(t_3^\e)=-\beta\e$.
Then during the time $t_3^\e-t_2^\e$ the trajectory rotates with respect to $\x(t_2^\e),\y(t_2^\e)$ of an angle
\[
\frac{t_3^\e-t_2^\e}{\e^2}=\frac{32}{27}\beta^2.
\]
Following the same steps as before, the solution of the system in the second transition zone is
\begin{equation}
\begin{cases}
\x=\sqrt{\rho}\cos\phi,\\
\y=\sqrt{\rho}\sin\phi,\\
\z=-\frac{2\beta\e}{1+\exp\left( -\frac{27}{8\beta^2\e^2} (t-t_3^\e,)\right)},
\end{cases}
\end{equation}
where
\[
\rho(t)=\rho(t_3^\e)e^{\frac{27}{8\beta^2\e^2}(t-t_3^\e)}\left(\frac{1+\exp\left(-\frac{27}{8\beta^2\e^2}(t-t_3^\e)\right)}{2}\right)^2,
\]
\[
\phi(t)=\phi_0-\frac{16}{9}\ln\left( \frac{1+\exp\left(-\frac{27}{8\beta^2\e^2}(t-t_3^\e)\right)}{2} \right)-\frac{2}{\beta^2\e^2}(t-t_3^\e),
\]
up to the time $t_4^\e:=t_3^\e+\frac{8\beta^2\e^2}{27}\ln(2)$, where $\phi_0=\theta_0+\bar{\theta}+\frac{32}{27}\beta^2$.
Then, note that $\phi(t_4^\e)=-\theta(t_2^\e)=-\bar{\theta}$ and the new initial datum for the ODE system is
\vspace{0.3cm} 
$$
\left(\sqrt{\frac{x^2+y^2}{z}}\sqrt{\frac{4}{3}\beta\e}\cos\left( \theta_0+\frac{32}{27}\beta^2\right)\; , \;\sqrt{\frac{x^2+y^2}{z}}\sqrt{\frac{4}{3}\beta\e}\sin\left( \theta_0+\frac{32}{27}\beta^2\right)\; , \; -\frac{4}{3}\beta\e\right).
$$
\vspace{0.3cm}
\\
For time larger than $t_4^\e$, the flow continues as 
\begin{equation}
\begin{cases}
\x=\sqrt{\frac{x^2+y^2}{z}}\sqrt[4]{4(t-t_4^\e)+\frac{16}{9}\beta^2\e^2}\cos\left( \theta_0+\frac{32}{27}\beta^2\right),\\
\y=\sqrt{\frac{x^2+y^2}{z}}\sqrt[4]{4(t-t_4^\e)+\frac{16}{9}\beta^2\e^2}\sin\left( \theta_0+\frac{32}{27}\beta^2\right),\\
\z=-\sqrt{4(t-t_4^\e)+\frac{16}{9}\beta^2\e^2}.
\end{cases}
\end{equation}
In conclusion, to find the solution $X^\Theta$ in the limit, we have to choose the parameter $\beta$ as
\[
\beta=\sqrt{\frac{27}{32}\Theta}.
\]
\\
$\underline{\mathit{Step \ 3}}$\\
\\
In this step we prove the convergence of the flows. \\ First we know that 
$$
X^\e(t,\mathbf{x})=X^\Theta(t,\mathbf{x})\hspace{0.7cm}
\forall \mathbf{x}\in P^-_\e\cup\left(\R^3\setminus\tilde{P}^\e\right), \ \forall t \in[0,T].
$$
We prove only the convergence for $x\in P^+_\e$, since the same argument works in $T^+_\e\cup C_\e\cup T_\e^-$. First of all, we have that
$$
X^\e(t,\mathbf{x})=X^\Theta(t,\mathbf{x}) \hspace{1cm} \forall t\in[0,t_1^\e], \  \forall \mathbf{x}\in P^+_\e.
$$
Then the trajectories $X^\e$ and $X^\Theta$ enter the approximated region and exit from it after a different amount of time, namely
$$
\Delta t_\e=(2+\ln 2)\frac{16}{27}\beta^2\e^2, \hspace{1cm} \Delta t_\Theta=\frac{8}{9}\beta^2\e^2.
$$
Since for $t\in[t_1^\e, t_1^\e+\Delta t_\Theta]$ both $X^\e(t,\mathbf{x})$ and $X^\Theta(t,\mathbf{x})$ are in $T^+_\e\cup C_\e\cup T^-_\e$, we have
$$
|X^\e(t,\mathbf{x})-X^\Theta(t,\mathbf{x})|\leq C\sqrt{\e}, \hspace{1cm} \forall t\in[t_1^\e, t_1^\e+\Delta t_\Theta], \ \forall \mathbf{x}\in P^+_\e.
$$
\\
For $t\in[t_1^\e+\Delta t_\Theta,t_1^\e+\Delta t_\e]$ the flow $X^\e$ is still in $C_\e$ while $X^\Theta$ lies in $P^-_\e$. Since 
$$
X_3^\Theta(t_4^\e,z)=-\frac{4}{3}\beta\e\sqrt{\frac{5+4\ln{2}}{3}},
$$
we have that 
$$
|X^\e(t,\mathbf{x})-X^\Theta(t,\mathbf{x})|\leq C\sqrt{\e}, \hspace{1cm} \forall t\in [t_1^\e,t_4^\e], \ \forall \mathbf{x}\in P^+_\e.
$$
\\
Then, for $t\geq t_4^\e$, the flow $X^\e$ exits the approximated region at the same point as the flow $X^\Theta$ and it can be written as
$$X^\e(t,\mathbf{x})=X^\Theta(t-\Delta_\e,\mathbf{x}),$$
where $\Delta_\e={O(\e^2)}$ is such that $t^\e_4=\frac{z^2}{4}+\Delta_\e$. So for $t\geq t_4^\e$, we estimate the difference $X^\e-X^\Theta$ component by component:
\begin{itemize}
\item for the third component we have
\begin{align*}
|X_3^\e(t,\mathbf{x})-X^\Theta_3(t,\mathbf{x})|& =|X_3^\Theta(t-\Delta_\e,z)-X^\Theta_3(t,z)| \\ &=|\sqrt{4t-z^2}-\sqrt{4(t-\Delta_\e)-z^2}| \\ &=
\frac{4\Delta_\e}{|\sqrt{4t-z^2}+\sqrt{4(t-\Delta_\e)-z^2}|} \\
&\leq \frac{4\Delta_\e}{\sqrt{4t-z^2}}\leq 2\sqrt{\Delta_\e}\leq C\e,
\end{align*}
\item for $i\in \{1,2\}$ we have
\begin{align*}
|X_i^\e(t,\mathbf{x})-X^\Theta_i(t,\mathbf{x})|& =|X_i^\Theta(t-\Delta_\e,\mathbf{x})-X^\Theta_i(t,\mathbf{x})| \\ &\leq \sqrt{\frac{x^2+y^2}{z}}\left|\sqrt[4]{4t-z^2}-\sqrt[4]{4(t-\Delta_\e)-z^2}\right| \\ &\leq \displaystyle\left|
\frac{\sqrt{4t-z^2}-\sqrt{4(t-\Delta_\e)-z^2}}{\sqrt[4]{4t-z^2}+\sqrt[4]{4(t-\Delta_\e)-z^2}}\right|\\ &\leq \frac{4\Delta_\e}{\sqrt[4]{4t-z^2}\sqrt{4t-z^2}}\leq C\sqrt[4]{\Delta_\e}\leq C\sqrt{\e}.
\end{align*}
\end{itemize} 
Note that in the previous estimate we have used the condition $x^2+y^2\leq z$. In conclusion, we have
$$
\sup_{t\in[0,T]}\sup_{\mathbf{x}\in\R^3}|X^\e(t,\mathbf{x})-X^\Theta(t,\mathbf{x})|<C\sqrt{\e},
$$
which gives the desired convergence.\\
\\
$\underline{\mathit{Step \ 4}}$\\
\\
{ In this step we prove the convergence of the inverse of the flows.\\
First of all, note that $X_{\e}^{\Theta}(t,\cdot)^{-1}=X_\e^\Theta(-t,\cdot)$, while $X^{\Theta}(t,\cdot)^{-1}$ is given by the following: if $(x,y,z)\in P^+,t\geq 0$ then
$$
X^{\Theta}(t,\cdot)^{-1}(x,y,z)=
\left(\frac{x}{\sqrt{z}}\sqrt[4]{z^2+4t},\frac{y}{\sqrt{z}}\sqrt[4]{z^2+4t},\sqrt{z^2+4t}\right).
$$
If $(x,y,z)\in P^-$ and $t\in \left[0,\frac{z^2}{4}\right]$,
$$
X^{\Theta}(t,\cdot)^{-1}(x,y,z)=\left( \frac{x}{\sqrt{-z}}\sqrt[4]{z^2-4t},\frac{y}{\sqrt{-z}}\sqrt[4]{z^2-4t},-\sqrt{z^2-4t}\right),
$$
while for $(x,y,z)\in P^-$ and $t\geq\frac{z^2}{4}$, $X^{\Theta}(t,\cdot)^{-1}(x,y,z)$ is given by
$$
\left( (x\cos\Theta+y\sin\Theta)\frac{\sqrt[4]{4t-z^2}}{\sqrt{-z}},(-x\sin\Theta+y\cos\Theta)\frac{\sqrt[4]{4t-z^2}}{\sqrt{-z}}, \sqrt{4t-z^2}\right).
$$
Finally $X^{\Theta}(t,\cdot)^{-1}(x,y,z)=(x,y,z)$ otherwise. Arguing as in the previous step, it is enough to prove the converngence in the region $P^-_\e$. We have that 
$$
X_{\e}^{\Theta}(t,\cdot)^{-1}(\mathbf{x})=X^{\Theta}(t,\cdot)^{-1}(\mathbf{x}) \hspace{1cm} \forall t\in[0,t_1^\e], \  \forall \mathbf{x}\in P^-_\e,
$$
where $t_1^\e=\sqrt{\frac{z^2}{4}-\frac{4}{9}\beta^2}$. Then the trajectories $\left(X^\e\right)^{-1}$ and $\left(X^\Theta\right)^{-1}$ enter the approximated region and exit from it after a different amount of time, namely
$$
\Delta t_\e=(2+\ln 2)\frac{16}{27}\beta^2\e^2, \hspace{1cm} \Delta t_\Theta=\frac{8}{9}\beta^2\e^2.
$$
Since for $t\in[t_1^\e, t_1^\e+\Delta t_\Theta]$ both $X^\e(t,\cdot)^{-1}(\mathbf{x})$ and $X^\Theta(t,\cdot)^{-1}(\mathbf{x})$ are in $T^+_\e\cup C_\e\cup T^-_\e$, we have
$$
|X^\e(t,\cdot)^{-1}(\mathbf{x})-X^\Theta(t,\cdot)^{-1}(\mathbf{x})|\leq C\sqrt{\e}, \hspace{1cm} \forall t\in[t_1^\e, t_1^\e+\Delta t_\Theta], \ \forall \mathbf{x}\in P^-_\e.
$$
\\
For $t\in[t_1^\e+\Delta t_\Theta,t_1^\e+\Delta t_\e]$ the flow $\left(X^\e\right)^{-1}$ is still in $C_\e$ while $\left(X^\Theta\right)^{-1}$ lies in $P^+_\e$. Note that the flow $\left(X^\e\right)^{-1}$ exit from the approximated region at time $t_4^\e=\frac{z^2}{4}+\frac{16}{27}\beta^2\e^2\log(2)+\frac{20}{27}\beta^2\e^2$ and $\left(X^\Theta\right)^{-1}$ is in
$$
X_3^\Theta(t_4^\e,\cdot)^{-1}(z)=\frac{4}{3}\beta\e\sqrt{\frac{5+4\ln{2}}{3}}.
$$
So we have that 
$$
|X^\e(t,\cdot)^{-1}(\mathbf{x})-X^\Theta(t,\cdot)^{-1}(\mathbf{x})|\leq C\sqrt{\e}, \hspace{1cm} \forall t\in [t_1^\e,t_4^\e], \ \forall \mathbf{x}\in P^-_\e.
$$
Then, for $t\geq t_4^\e$, the flow $\left(X^\e\right)^{-1}$ exits the approximated region at the same point as the flow $\left(X^\Theta\right)^{-1}$ and it can be written as
$$X^\e(t,\cdot)^{-1}(\mathbf{x})=X^\Theta(t-\Delta_\e,\cdot)^{-1}(\mathbf{x}),$$
where $\Delta_\e=O(\e^2)$ is such that $t^\e_4=\frac{z^2}{4}+\Delta_\e$. So for $t\geq t_4^\e$, we estimate the difference $\left(X^\e\right)^{-1}-\left(X^\Theta\right)^{-1}$ component by component:
\begin{itemize}
\item for the third component we have
\begin{align*}
|X_3^\e(t,\cdot)^{-1}(\mathbf{x})-X^\Theta_3(t,\cdot)^{-1}(\mathbf{x})|& =|X_3^\Theta(t-\Delta_\e,\cdot)^{-1}(z)-X^\Theta_3(t,\cdot)^{-1}(z)| \\ &=\left|\sqrt{4(t-\Delta_\e)-z^2}-\sqrt{4t-z^2}\right| \\ &=
\frac{4\Delta_\e}{|\sqrt{4t-z^2}+\sqrt{4(t-\Delta_\e)-z^2}|} \\
&\leq \frac{4\Delta_\e}{\sqrt{4t-z^2}}\leq 2\sqrt{\Delta_\e}\leq C\e,
\end{align*}
\item for $i\in \{1,2\}$ we have
\begin{align*}
|X_i^\e(t,\cdot)^{-1}(\mathbf{x})-X^\Theta_i(t,\cdot)^{-1}(\mathbf{x})|& =|X_i^\Theta(t-\Delta_\e,\cdot)^{-1}(\mathbf{x})-X^\Theta_i(t,\cdot)^{-1}(\mathbf{x})| \\ &\leq \sqrt{\frac{x^2+y^2}{-z}}\left|\sqrt[4]{4(t-\Delta_\e)-z^2}-\sqrt[4]{4t-z^2}\right| \\ &\leq \displaystyle\left|\frac{
\sqrt{4(t-\Delta_\e)-z^2}-\sqrt{4t-z^2}}{\sqrt[4]{4t-z^2}+\sqrt[4]{4(t-\Delta_\e)-z^2}}\right| \\ &\leq \frac{4\Delta_\e}{\sqrt[4]{4t-z^2}\sqrt{4t-z^2}}\leq C\sqrt[4]{\Delta_\e}\leq C\sqrt{\e}.
\end{align*}
\end{itemize} 
Note that in the previous estimate we have used the condition $x^2+y^2\leq -z$. In conclusion, we have
$$
\sup_{t\in[0,T]}\sup_{\mathbf{x}\in\R^3}|X^\e(t,\cdot)^{-1}(\mathbf{x})-X^\Theta(t,\cdot)^{-1}(\mathbf{x})|<C\sqrt{\e},
$$
which gives the desired convergence.
}
\\
\\
$\underline{\mathit{Step \ 5}}$\\
\\
In this step we check the regularity of $b_\e$. It is easy to verify that $b_\e$ is locally bounded and $$\|\nabla b_\e\|_\infty\leq C \e^{-5/2} $$ inside $\tilde{P}_\e$ up to the boundary, so $b_\e$ is Lipschitz inside $\tilde{P}_\e$ for fixed $\e$. Furthermore $b_\e=0$ in $\R^3\setminus \tilde{P}^\e$ and the jump across the surface $\partial\tilde{P}^\e$ is controlled by $C \e^{-5/2}$ implying $b_\e\in BV_{\mathrm{loc}}(\R^3)$.
We can easily prove that $\dive b_\e=0$ inside $\tilde{P}_\e$ and that $b_\e$ is tangent to $\partial\tilde{P}_\e$, hence it is divergence-free in the sense of distributions in the whole space.
The growth condition follows easily from the fact that the limit vector field $b$ verifies it.
\end{proof}

\section{Proof of the main theorems}
\vspace{0.2cm}
In this section we give the proofs of the main theorems stated in the introduction, which we restate here for the reader's convenience.
\begin{thm2}
There exists a divergence free vector field $b\in L_{\mathrm{loc}}^{p}(\R^{3})$ with $p\in [1,\frac{4}{3}]$, and a sequence of divergence-free vector fields $b_n\in C^\infty(\R^{3})$  such that $b_n\to b$ strongly in $L_{\mathrm{loc}}^{p}(\R^{3})$ and the uniquely defined sequence $X^n$ of flows of $b_n$ does not converge, but has at least two different subsequences along converging in $L^\infty((0,T);L^1_{\mathrm{loc}}(\R^3))$ to two different flows.
\end{thm2}
\begin{proof}
Let $b$ be the vector field defined in $(\ref{eq:b})$ and let $\Theta,\Phi \in (0,2\pi]$ with $\Theta\neq\Phi$. From Proposition \ref{prop:main} there exist $b_\e^\Theta,b_\e^\Phi\in BV_{\mathrm{loc}}(\R^3)$ and $X^\Theta,X^\Phi\in C([0,T];L^1_{\mathrm{loc}}(\R^3))$ with the following properties. First, it holds that
\begin{equation}\label{1}
\begin{aligned}
& b_\e^\Theta \longrightarrow b \ \mathrm{in} \ L^1_{\mathrm{loc}}(\R^3), \ \mathrm{as} \ \e\to 0, \\
& b_\e^\Phi \longrightarrow b \ \mathrm{in} \ L^1_{\mathrm{loc}}(\R^3), \ \mathrm{as} \ \e\to 0.
\end{aligned}
\end{equation}
Moreover, by denoting with $X^\Theta_\e,X^\Phi_\e$ the unique regular Lagrangian flows of $b_\e^\Theta,b_\e^\Phi$, it holds that
\begin{equation}\label{2}
\begin{aligned}
& X_\e^\Theta \longrightarrow X^\Theta \ \mathrm{in} \ L^1((0,T);L^1_{\mathrm{loc}}(\R^3), \ \mathrm{as} \ \e\to 0, \\
& X_\e^\Phi \longrightarrow X^\Phi \ \mathrm{in} \ L^1((0,T);L^1_{\mathrm{loc}}(\R^3), \ \mathrm{as} \ \e\to 0.
\end{aligned}
\end{equation}
Let $b_{\e,l}^\Theta,b_{\e,k}^\Phi\in C^\infty(\R^3)$ be mollifications of $b_\e^\Theta,b_\e^\Phi$. Since $b_\e^\Theta,b_\e^\Phi$ are in $BV_{\mathrm{loc}}(\R^3)$ for fixed $\e$, by using Theorem $\ref{teo:stab}$ it follows that, for $\e>0$ fixed
\begin{equation}\label{3}
\begin{aligned}
& X_{\e,l}^\Theta \longrightarrow X^\Theta_\e \ \mathrm{in} \ L^1((0,T);L^1_{\mathrm{loc}}(\R^3), \ \mathrm{as} \ l\to \infty, \\
& X_{\e,k}^\Phi \longrightarrow X^\Phi_\e \ \mathrm{in} \ L^1((0,T);L^1_{\mathrm{loc}}(\R^3), \ \mathrm{as} \ k\to \infty.
\end{aligned}
\end{equation}
where $X_{\e,l}^\Theta,X_{\e,k}^\Phi$ denote the smooth flows of $b_{\e,l}^\Theta,b_{\e,k}^\Phi$ respectively. By using $\eqref{2},\eqref{3}$ and a simple diagonal argument there exist $\e_i,l_i,\e_j,k_j$ with $i,j\in\N$ such that
\begin{equation*}
\begin{aligned}
& X_{\e_i,l_i}^\Theta \longrightarrow X^\Theta \ \mathrm{in} \ L^1((0,T);L^1_{\mathrm{loc}}(\R^3), \ \mathrm{as} \ i\to \infty, \\
& X_{\e_j,k_j}^\Phi \longrightarrow X^\Phi \ \mathrm{in} \ L^1((0,T);L^1_{\mathrm{loc}}(\R^3), \ \mathrm{as} \ j\to \infty.
\end{aligned}
\end{equation*}
Finally, since both $b_{\e,l}^\Theta,b_{\e,k}^\Phi$ strongly converge in $L^1_{\mathrm{loc}}(\R^3)$ to $b$, by merging $b_{\e_i,l_i}^\Theta,b_{\e_j,k_j}^\Phi$ and appropriately renaming the indexes we can infer that there exists $\{b_n\}_n$ as claimed in the statement of the theorem.
\end{proof}
We now move to the proof of Theorem \ref{teo:main1}. 
\begin{thm1}
There exist an autonomous divergence-free vector field $b\in L_{\mathrm{loc}}^{p}(\R^{3})$ with $p\in [1,\frac{4}{3}]$, and a sequence of divergence-free vector fields $b_n\in C^\infty(\R^{3})$ converging to $b$ strongly in $L^p_{\mathrm{loc}}(\R^3)$ such that the following holds. Let $u_0\in \mathcal{I}\cap L^\infty(\R^3)$ be a given inital datum, then there exist subsequences $n_i$ and $n_j$ such that the sequences $u_{n_i}$ and $u_{n_j}$, solutions of \eqref{eq:tee} with initial datum $u_0$, converge in $L^\infty((0,T);L^\infty(\R^3)-w*)\cap L^\infty((0,T);L^1_{\mathrm{loc}}(\R^3))$ to two different limits, which are bounded distributional solutions of \eqref{eq:te}.
\end{thm1}
\begin{proof}
Let $b$ be the vector field defined in $(\ref{eq:b})$ and let $\Theta,\Phi \in (0,2\pi]$ with $\Theta\neq\Phi$. From Proposition \ref{prop:main} there exist $b_\e^\Theta,b_\e^\Phi\in BV_{\mathrm{loc}}(\R^3)$ such that
\begin{equation}
\begin{aligned}
& b_\e^\Theta \longrightarrow b \ \mathrm{in} \ L^1_{\mathrm{loc}}(\R^3), \ \mathrm{as} \ \e\to 0, \\
& b_\e^\Phi \longrightarrow b \ \mathrm{in} \ L^1_{\mathrm{loc}}(\R^3), \ \mathrm{as} \ \e\to 0.
\end{aligned}
\end{equation}
Let $u_0\in \mathcal{I}\cap L^\infty(\R^3)$ and consider the Cauchy problems
\begin{equation}\label{eq:th}
\begin{cases}
\partial_t u^\e_\Theta+b_\e^\Theta\cdot\nabla u^\e_\Theta=0,\\
u^\e_\Theta|_{t=0}=u_0,
\end{cases}
\end{equation}
\begin{equation}\label{eq:ph}
\begin{cases}
\partial_t u^\e_\Phi+b_\e^\Phi\cdot\nabla u^\e_\Phi=0,\\
u^\e_\Phi|_{t=0}=u_0.
\end{cases}
\end{equation}
Since $b_\e^\Theta, b_\e^\Phi$ verify the hypotesis of Theorem \ref{teo:am} for every fixed $\e$, the solutions of \ref{eq:th} and \ref{eq:ph} are unique and they are given respectively by the formulas
\begin{equation}
\begin{aligned}
& u^\e_\Theta(t,x)=u_0((X^\Theta_\e(t,\cdot)^{-1}(x)), \\
& u^\e_\Phi(t,x)=u_0((X^\Phi_\e(t,\cdot)^{-1}(x)),
\end{aligned}
\end{equation}
where $X^\Theta_\e,X^\Phi_\e$ are the unique Regular Lagrangian Flows of $b_\e^\Theta,b_\e^\Phi$.
Then $u^\e_\Theta$ converge uniformly on compact sets to $u_\Theta:=u_0((X^\Theta)^{-1})$, since
\begin{align*}
\sup_{t\in[0,T]}\sup_{B_R} |& u_\Theta^\e(t,\cdot)-  u_\Theta(t,\cdot)| \\
& =\sup_{t\in[0,T]}\sup_{B_R} |u_{0}((X_{\e}^{\Theta}(t,\cdot)^{-1})-u_0((X^\Theta(t,\cdot)^{-1})| \\ 
&\leq \|\nabla u_0\|_{L^\infty(B_R)} \sup_{t\in[0,T]}\sup_{\R^3}|X_{\e}^{\Theta}(t,\cdot)^{-1}-X^\Theta(t,\cdot)^{-1}| \\
& \leq C \sqrt{\e}.
\end{align*}
Here $B_R$ is a closed ball of radius $R>0$, $u_0\in C^\infty(\R^3)$ so it is Lipschitz on compact sets and the backward flow $X_{\e}^{\Theta}(t,\cdot)^{-1}$ converges uniformly to $X^\Theta(t,\cdot)^{-1}$ with rate $\sqrt{\e}$, see Step 4 in the proof of Proposition \ref{prop:main}.\\ The same convergence holds for $u^\e_\Phi$ towards $u_\Phi:=u_0((X^\Phi)^{-1})$.\\
Let $b_{\e,l}^\Theta,b_{\e,k}^\Phi\in C^\infty(\R^3)$ be regularizations of $b_\e^\Theta,b_\e^\Phi$. Since $b_\e^\Theta,b_\e^\Phi$ are in $BV_{\mathrm{loc}}(\R^3)$ for fixed $\e>0$, using Theorem $\ref{teo:stab}$ it follows that
\begin{equation*}
\begin{aligned}
& u^{\e,l}_\Theta \longrightarrow u_\Theta^\e \ \mathrm{in} \ L^\infty([0,T];L^\infty(\R^3)-w*)\cap L^\infty([0,T];L^1_{\mathrm{loc}}(\R^3)), \ \mathrm{as} \ l\to \infty, \\
& u^{\e,k}_\Phi \longrightarrow u_\Phi^\e \ \mathrm{in} \ L^\infty([0,T];L^\infty(\R^3)-w*)\cap L^\infty([0,T];L^1_{\mathrm{loc}}(\R^3)), \ \mathrm{as} \ k\to \infty.
\end{aligned}
\end{equation*}
Arguing as in the proof of Theorem \ref{teo:main2}, by a diagonal argument we can infer that there exist $\e_i,l_i,\e_j,k_j$ with $i,j\in\N$ such that
\begin{equation*}
\begin{aligned}
& u^{\e_i,l_i}_\Theta \longrightarrow u_\Theta \ \mathrm{in} \ L^\infty([0,T];L^\infty(\R^3)-w*)\cap L^\infty([0,T];L^1_{\mathrm{loc}}(\R^3)), \ \mathrm{as} \ i\to \infty, \\
& u^{\e_j,k_j}_\Phi \longrightarrow u_\Phi \ \mathrm{in} \ L^\infty([0,T];L^\infty(\R^3)-w*)\cap L^\infty([0,T];L^1_{\mathrm{loc}}(\R^3)), \ \mathrm{as} \ j\to \infty.
\end{aligned}
\end{equation*}
Since both $b_{\e,l}^\Theta,b_{\e,k}^\Phi$ strongly converge in $L^1_{\mathrm{loc}}(\R^3)$ to $b$, by merging $b_{\e_i,l_i}^\Theta,b_{\e_j,k_j}^\Phi$ and appropriately renaming the indexes we can infer that there exists $\{b_n\}_n$ as claimed in the statement of the theorem. Indeed, considering an initial datum $u_0\in \mathcal{I}\cap L^\infty(\R^3)$, it holds that $u_\Theta\neq u_\Phi$. Thus, they are actually two different solutions of (\ref{eq:te}). 
\end{proof}
Finally we prove the following corollary.
\begin{cor}\label{cor:visc}
There exist an autonomous divergence-free vector field $b\in L_{\mathrm{loc}}^{p}(\R^{3})$ with $p\in [1,\frac{4}{3}]$, a sequence of divergence-free vector fields $b_m\in C^\infty(\R^{3})$ converging to $b$ strongly in $L^p_{\mathrm{loc}}(\R^3)$, and an infinitesimal sequence $\{\delta_m\}_m\subset\R$ such that the following holds. Let $u_0\in \mathcal{I}\cap L^\infty(\R^3)$ be a given inital datum, and consider the Cauchy problem
\begin{equation}\label{eq:vvd}
\begin{cases}
\partial_t u_m +b_m\cdot\nabla u_m=\delta_m \Delta u_m,\\
u_m(0,\cdot)=u_0.
\end{cases}
\end{equation}
Then, there exist subsequences $m_i$ and $m_j$ such that the sequences $u_{m_i}$ and $u_{m_j}$, solutions of \eqref{eq:vvd}, converge in $L^\infty((0,T);L^\infty(\R^3)-w*)\cap L^\infty((0,T);L^1_{\mathrm{loc}}(\R^3))$ to two different limits, which are bounded distributional solutions of \eqref{eq:te}. 
\end{cor}
\begin{proof}
Let $b$ be the vector field defined in \eqref{eq:b} and let $\Theta,\Phi \in (0,2\pi]$ with $\Theta\neq\Phi$. Let $b_n, u_{n_i},u_{n_j}$ be given by Theorem \ref{teo:main1}, we have that
\begin{equation}
\begin{aligned}
b_n \longrightarrow b \ \ \ \ \ &\mbox{in } L^1_{\mathrm{loc}}(\R^3), \mbox{ as } n\to \infty, \\
u_{n_i} \longrightarrow u_\Theta \ & \mbox{in } L^\infty((0,T);L^\infty(\R^3)-w*)\cap L^\infty([0,T];L^1_{\mathrm{loc}}(\R^3)),  \mbox{ as } i\to \infty,\\
u_{n_j} \longrightarrow u_\Phi \ & \mbox{in } L^\infty((0,T);L^\infty(\R^3)-w*)\cap L^\infty([0,T];L^1_{\mathrm{loc}}(\R^3)),  \mbox{ as } j\to \infty.
\end{aligned}
\end{equation}
Let us consider for $\delta>0$ the following Cauchy problems
\begin{equation}\label{eq:vvndeltai}
\begin{cases}
\partial_t u^\delta_{n_i}+b_{n_i}\cdot\nabla u^\delta_{n_i}=\delta\Delta u^\delta_{n_i},\\
u^\delta_{n_i}(0,\cdot)=u_0,
\end{cases}
\end{equation}
\begin{equation}\label{eq:vvndeltaj}
\begin{cases}
\partial_t u^\delta_{n_j}+b_{n_j}\cdot\nabla u^\delta_{n_j}=\delta\Delta u^\delta_{n_j},\\
u^\delta_{n_j}(0,\cdot)=u_0.
\end{cases}
\end{equation}
For every fixed $\delta>0$, by the smoothness of $b_{n_i}$ and $b_{n_j}$ we can infer that there exists a unique solution $u^\delta_{n_i}$ and $u^\delta_{n_j}$ respectively of \eqref{eq:vvndeltai} and \eqref{eq:vvndeltaj}. Moreover, for fixed $i$ and $j$, by letting $\delta\to 0$ we have
\begin{equation}
u^\delta_{n_i} \longrightarrow u_{n_i} \ \mathrm{in} \ L^\infty((0,T);L^\infty(\R^3)-w*)\cap  L^\infty([0,T];L^1_{\mathrm{loc}}(\R^3)),
\end{equation}
\begin{equation}
u^\delta_{n_j} \longrightarrow u_{n_j} \ \mathrm{in} \ L^\infty((0,T);L^\infty(\R^3)-w*)\cap L^\infty([0,T];L^1_{\mathrm{loc}}(\R^3)).
\end{equation}
By a diagonal argument we can infer that there exist $\delta_l, n_{i_l}, \delta_k, n_{j_k}$ such that
\begin{equation}
u_{n_{i_l}}^{\delta_l} \longrightarrow u_\Theta \ \mathrm{in} \ L^\infty((0,T);L^\infty(\R^3)-w*)\cap L^\infty([0,T];L^1_{\mathrm{loc}}(\R^3)), \ \mathrm{as} \ l\to\infty,
\end{equation}
\begin{equation}
u_{n_{j_k}}^{\delta_k} \longrightarrow u_\Phi \ \mathrm{in} \ L^\infty((0,T);L^\infty(\R^3)-w*)\cap L^\infty([0,T];L^1_{\mathrm{loc}}(\R^3)), \ \mathrm{as} \ k\to\infty.
\end{equation}
Finally, by merging the sequences and appropriately renaming the indexes we obtain the result.
\end{proof}

\subsection*{Acknowledgments}
This research has been supported by the ERC Starting Grant 676675 FLIRT. We thank the anonymous referee and L\`aszl\`o Sz\`ekelyhidi for the careful reading of the paper and for several useful comments.

\end{document}